\documentclass[12pt, reqno]{amsart}
\usepackage{amsmath, amsthm, amscd, amsfonts, amssymb, graphicx, color}
\usepackage[bookmarksnumbered, colorlinks, plainpages]{hyperref}
\input{mathrsfs.sty}

\newtheorem{theorem}{Theorem}[section]
\newtheorem{lemma}[theorem]{Lemma}
\newtheorem{proposition}[theorem]{Proposition}
\newtheorem{corollary}[theorem]{Corollary}
\theoremstyle{definition}

\theoremstyle{remark}
\newtheorem{remark}[theorem]{Remark}
\numberwithin{equation}{section}

\begin{document}

\title[Majorization and range inclusion of operators]{On majorization and range inclusion of operators on Hilbert $C^*$-modules}

\author[X. Fang]{Xiaochun Fang$^1$}
\address{$^1$Department of Mathematics, Tongji University, Shanghai 200092, PR China}
\email{xfang@tongji.edu.cn}
\author[M. S. Moslehian]{Mohammad Sal Moslehian$^2$}
\address{$^2$Department of Pure Mathematics, Ferdowsi University of Mashhad, P.O. Box 1159, Mashhad 91775, Iran}
\email{moslehian@um.ac.ir}
\author[Q. Xu]{Qingxiang Xu$^3$$^*$}
\address{$^3$Department of Mathematics, Shanghai Normal University, Shanghai 200234, PR China}
\email{qxxu@shnu.edu.cn}

\keywords{Hilbert $C^*$-module, majorization, range inclusion.\\
$^1$ Supported by the National Natural Science Foundation of China (11371279, 11511130012).\\
$^2$ Supported by a grant from the Iran National Science Foundation (No. 96006713).\\
$^3$ Supported by the National Natural Science Foundation of China (11671261).}

\begin{abstract} It is proved that for adjointable operators $A$ and $B$ between Hilbert $C^*$-modules, certain majorization conditions are always equivalent without any assumptions on
 $\overline{\mathcal{R}(A^*)}$, where $A^*$ denotes the adjoint operator of $A$ and $\overline{\mathcal{R}(A^*)}$ is the norm closure of the range of $A^*$. In the case that $\overline{{\mathcal R}(A^*)}$ is not orthogonally complemented,
 it is proved that there always exists an adjointable operator $B$ whose range is contained in that of $A$,
whereas the associated equation
$AX=B$ for adjointable operators is unsolvable.
 \end{abstract}

 \maketitle

\section{Introduction}\label{sec:pre}

Hilbert $C^*$-modules are generalizations of Hilbert spaces by allowing inner products to take values in some $C^{*}$-algebras instead of the field of complex numbers. Hilbert $C^*$-modules are useful tools in $AW^*$-algebra theory, theory of operator algebras, operator K-theory, group representation theory and theory of operator spaces. Some basic properties of Hilbert spaces are no longer valid in setting of Hilbert $C^*$-modules in their full generality. For example, a closed submodule may not be orthogonally complemented and a bounded linear operator between Hilbert $C^*$-modules may have no adjoint operator. Therefore, when we are studying Hilbert $C^*$-modules, it is always of interest under which conditions the results analogous to those for Hilbert spaces can be reobtained, as well as which more general situations might appear.

Douglas \cite{douglas} studied the equation $AX=B$ for bounded linear operators on Hilbert spaces, and gave his so-called Douglas theorem. This result extensively applied in division and quotients of operators, operator equations, operator range inclusions, generalized inverses, and operator inequalities. This is not the case in more general situations, e.g. in the framework of Banach spaces. There however are some variants of this theorem under some conditions in various settings, see e.g. \cite{BAR, hhw, POP, ROD, SEB} and references therein.

A generalization of Douglas theorem to the Hilbert $C^*$-module case was given in \cite{fang-yu-yao}, which can be stated as follows:
\begin{theorem}\label{thm:fang s result} {\rm \cite[Theorem 1.1]{fang-yu-yao}}\ Let $\mathfrak{A}$ be a $C^*$-algebra, $E,H$ and $K$ be Hilbert $\mathfrak{A}$-modules. Let $T\in {\mathcal L}(E,K)$ and $T^\prime\in {\mathcal L}(H,K)$ be such that $\overline{{\mathcal R}(T^*)}$ is orthogonally complemented. Then the following statements are equivalent:
\begin{enumerate}
\item[{\rm (i)}] $T^\prime (T^\prime)^*\le \lambda TT^*$ for some $\lambda>0$;
\item[{\rm (ii)}]There exists $\mu>0$ such that $\Vert (T^\prime)^*z\Vert\le \mu \Vert T^*z\Vert$, for any $z\in K$;
\item[{\rm (iii)}]There exists a solution $X\in {\mathcal L}(H,E)$ of the so-called Douglas equation $T^\prime =TX$;
\item[{\rm (iv)}]${\mathcal R}(T^\prime)\subseteq {\mathcal R}(T)$.
\end{enumerate}
\end{theorem}

Much attention has been paid since the publication of \cite{fang-yu-yao}, and Theorem~\ref{thm:fang s result} was cited in many literatures; see the references in \cite{Xu-Fang} as well as \cite{FOR}. It follows from the proof of \cite[Theorem~1.1]{fang-yu-yao} that
``(iv)$\Longrightarrow$ (ii)" and ``(ii)+(iv)$\Longrightarrow$ (iii)$\Longrightarrow$ (i)$\Longrightarrow$ (ii)" are true. Clearly,
``(iii)$\Longrightarrow$ (iv)" is valid, so conditions (iii) and (iv) in Theorem~\ref{thm:fang s result} are really equivalent if $\overline{{\mathcal R}(T^*)}$ is orthogonally complemented. Unfortunately, a counterexample is constructed in \cite{Xu-Fang} which indicates that the implication ``(i)$\Longrightarrow$ (iv)" is false even if $\overline{{\mathcal R}(T^*)}$ is orthogonally complemented. A gap is then contained in \cite{fang-yu-yao} to the proof of the implication ``(ii)$\Longrightarrow$ (i)", see also \cite{Moslehian}.

The purpose of this paper is to cover the gap mentioned above, and to investigate the
equivalence of the conditions above when $\overline{{\mathcal R}(T^*)}$ is not orthogonally complemented.
 Actually, we will prove, without taking account $\overline{{\mathcal R}(T^*)}$ is orthogonally complemented or not, conditions (i) and (ii) are always equivalent; see Corollary~\ref{cor:equivalence of i and ii}. It is remarkable that the same is not true for
 conditions (iii) and (iv). In fact, for any Hilbert $\mathfrak{A}$-module $E, K$, and any operator $T \in \mathcal{L}(E,K)$, a Hilbert $\mathfrak{A}$-module $G$ and an operator $S\in \mathcal{L}(G,K)$ can be introduced as \eqref{equ:defn of G and S}
 such that $\mathcal{R}(S)$ is contained in $\mathcal{R}(T)$ evidently, whereas the equation
$S=TX$ for $X\in {\mathcal L}(G,E)$ is solvable only if $\overline{{\mathcal R}(T^*)}$ is orthogonally complemented.

 The rest of this paper is arranged as follows. In Section~\ref{sec:equivalence of two majorization conditions}, we first recall some basic knowledge about Hilbert $C^*$-modules. Based on a slight modification of the proof of \cite[Lemma~3.4]{lance}, an inequality is clarified in Lemma~\ref{lem:lance p25} for the square of two positive elements in a $C^*$-algebra. As an application, the equivalence of two majorization conditions is proved in Theorem~\ref{thm:main thm-1}. The equivalence of conditions (i) and (ii) then follows directly without any assumptions on $\overline{{\mathcal R}(T^*)}$.
 In Section~\ref{sec:orthogonally complemented condition}, we focus on the study of the orthogonally complemented condition for $\overline{{\mathcal R}(T^*)}$. Two equivalent conditions are provided in
Theorem~\ref{thm:two orthogonally complemented conditions}, which indicate the mentioned unsolvbility of the equation $S=TX$ for $X \in {\mathcal L}(G,E)$ when
 $\overline{{\mathcal R}(T^*)}$ is not orthogonally complemented. Finally, an example is constructed in Proposition~\ref{prop:nonequivalence} where $T\in {\mathcal L}(H,K)$ and $T^\prime\in {\mathcal L}(K)$ are such that ${\mathcal R}(T^\prime)$ is contained in ${\mathcal R}(T)$ properly, whereas the equation $TX=T^\prime$ for $X\in\mathcal{L}(K,H)$ in unsolvable.

\section{Equivalence of two majorization conditions}\label{sec:equivalence of two majorization conditions}

Throughout the rest of this paper, $\mathfrak{A}$ is a $C^*$-algebra. An
inner-product $\mathfrak{A}$-module is a linear space $E$ which is a right
$\mathfrak{A}$-module, together with a map $(x,y)\to \big<x,y\big>:E\times E\to
\mathfrak{A}$ such that for any $x,y,z\in E, \alpha, \beta\in \mathbb{C}$ and
$a\in \mathfrak{A}$, the following conditions hold:
\begin{enumerate}
\item[(i)] $\langle x,\alpha y+\beta
z\rangle=\alpha\langle x,y\rangle+\beta\langle x,z\rangle;$

\item[(ii)] $\langle x, ya\rangle=\langle x,y\rangle a;$

\item[(iii)] $\langle y,x\rangle=\langle x,y\rangle^*;$

\item[(iv)] $\langle x,x\rangle\ge 0, \ \mbox{and}\ \langle x,x\rangle=0\Longleftrightarrow x=0.$
\end{enumerate}

An inner-product $\mathfrak{A}$-module $E$ which is complete with respect to
the induced norm ($\Vert x\Vert=\sqrt{\Vert \langle x,x\rangle\Vert}$ for
$x\in E$) is called a (right) Hilbert $\mathfrak{A}$-module.
A closed
submodule $F$ of a Hilbert $\mathfrak{A}$-module $E$ is said to be
orthogonally complemented if $E=F\oplus F^\bot$, where
$$F^\bot=\left\{x\in E : \langle x,y\rangle=0\ \mbox{for any }\ y\in
F\right\}.$$

Now suppose that $H$ and $K$ are two Hilbert $\mathfrak{A}$-modules, let ${\mathcal L}(H,K)$
be the set of operators $T:H\to K$ for which there is an operator $T^*:K\to
H$ such that $$\langle Tx,y\rangle=\langle x,T^*y\rangle \ \mbox{for any
$x\in H$ and $y\in K$}.$$ It is known that any element $T \in {\mathcal
L}(H,K)$ must be a bounded linear operator, which is also $\mathfrak{A}$-linear
in the sense that $T(xa)=(Tx)a$, for $x\in H$ and $a\in \mathfrak{A}$. We call ${\mathcal
L}(H,K)$ the set of adjointable operators from $H$ to $K$. For any
$T\in {\mathcal L}(H,K)$, the range and the null space of $T$ are denoted by
${\mathcal R}(T)$ and ${\mathcal N}(T)$, respectively. In case
$H=K$, ${\mathcal L}(H,H)$ which is abbreviated to ${\mathcal L}(H)$, is a
$C^*$-algebra. Let ${\mathcal L}(H)_+$ be the set of positive elements
in ${\mathcal L}(H)$, which can be characterized as follows:
\begin{lemma}\label{lem:positive operator}{{\rm \cite[Lemma 4.1]{lance}}} Let $H$ be a Hilbert $\mathfrak{A}$-module and $T\in {\mathcal L}(H)$. Then $T\in {\mathcal L}(H)_+$ if and only if
$\langle Tx, x \rangle\ge 0$ \ for all $x$ in $H$.
\end{lemma}

For any $C^*$-algebra $\mathfrak{B}$, let $\mathfrak{B}_+$ be the set of all positive elements of $\mathfrak{B}$.
Trivial as it is, the following lemma is stated for the sake of completeness.

\begin{lemma}\label{lem:C-star elementary}{\rm \cite[Proposition~1.3.5]{pedersen}}\ Let $\mathfrak{B}$ be a $C^*$-algebra and $x,y\in \mathfrak{B}_+$ be such that $x\le y$. Then $\Vert x\Vert\le \Vert y\Vert$ and $z^*xz\le z^*yz$ for
each $z \in \mathfrak{B}$.
\end{lemma}

Along the same line in the proof of \cite[Lemma~3.4]{lance}, we get a lemma as follows:
\begin{lemma}\label{lem:lance p25} Let $\mathfrak{B}$ be a $C^*$-algebra and $a,b\in \mathfrak{B}_+$ be such that
\begin{equation}\label{equ:norm ac bc inequality}\Vert ac\Vert\le \Vert bc\Vert \ \mbox{for all $c \in \mathfrak{B}_+$}.\end{equation}
Then $a^2\le b^2$.
\end{lemma}
\begin{proof}Without loss of generality, we may assume that $\Vert a\Vert\le 1$ and $\Vert b\Vert\le 1$. Suppose that the inequality $a^2\le b^2$ is not true. Then there exists $t_0\in \mbox{sp}(a^2-b^2)$ such that $t_0>0$, where sp$(a^2-b^2)$ denotes the spectrum of $a^2-b^2$.
It follows that
\begin{equation*}m=\max\left\{t: t\in\mbox{sp}(a^2-b^2)\right\}>0.
\end{equation*}

Let $f$ be any continuous real-valued function defined on the real line such that
\begin{eqnarray*}&&0\le f(t)\le 1 \ \mbox{for all $t\in (-\infty, +\infty)$},\\
&&f(t)=0\ \mbox{for all $t\in (-\infty, \frac{m}{2}]$},\\
&&f(t)=1\ \mbox{for all $t\in [m, +\infty)$}.
\end{eqnarray*}
Put $c=f(a^2-b^2)$. Then $c\in \mathfrak{B}_+$ since $f$ is nonnegative. By use of the functional calculus \cite[Sec.\,1.1]{pedersen}, we have
\begin{equation}\label{equ:equals m}\Vert c(a^2-b^2)c\Vert=\max\left\{tf(t)^2: t\in \mbox{sp}(a^2-b^2)\right\}=m>0.\end{equation}
Since $t>\frac{m}{2}$ whenever $f(t)\ne 0$, it follows again by use of the functional calculus that
\begin{equation}\label{equ:inequality-2}c(a^2-b^2)c\ge \frac{m}{2}\, c^2.\end{equation}
Note that $c b^2 c$ is self-adjoint, so there exists a state $\rho$ acting on $\mathfrak{B}$ such that $\rho(c b^2 c)=\Vert c b^2 c\Vert$. Then by
\eqref{equ:inequality-2} and the assumption $\Vert b\Vert\le 1$, we have
\begin{equation*}\Vert ca^2 c\Vert\ge \rho(ca^2 c)\ge \rho(cb^2 c)+\frac{m}{2}\,\rho(c^2)\ge \frac{m+2}{2}\rho(cb^2c)=\frac{m+2}{2}\, \Vert c b^2 c\Vert.
\end{equation*}
This shows that if $cb^2c\ne 0$, then
$\Vert ca^2 c\Vert>\Vert cb^2 c\Vert$. On the other hand, if $cb^2c=0$, then it follows from \eqref{equ:equals m}
that $\Vert ca^2 c\Vert >0$. So in either case we have
\begin{equation*}\Vert ac\Vert^2=\Vert ca^2 c\Vert>\Vert cb^2 c\Vert=\Vert bc\Vert^2,\end{equation*}
in contradiction to the assumption that $\Vert ac\Vert\le \Vert bc\Vert.$
\end{proof}

Now, we state the main result of this section as follows:
\begin{theorem}\label{thm:main thm-1} Let $E, H$ and $K$ be Hilbert $\mathfrak{A}$-modules, $T\in {\mathcal L}(E,K)$ and $S\in {\mathcal L}(H,K)$. Then the following two statements are equivalent:
\begin{enumerate}
\item[{\rm (i)}] $TT^*\le SS^*$;
\item[{\rm (ii)}] $\Vert T^*x\Vert \le \Vert S^*x\Vert$ for all $x \in K$.
\end{enumerate}
\end{theorem}
\begin{proof}``(i)$\Longrightarrow$(ii)" follows directly from Lemmas~\ref{lem:positive operator} and \ref{lem:C-star elementary}.

``(ii)$\Longrightarrow$(i)": Given any $x\in K$, let $a=\langle TT^*x,x\rangle, b=\langle SS^*x,x\rangle$. Then $a,b\in \mathfrak{A}_+$, and for any $c\in \mathfrak{A}_+$, it holds that
\begin{eqnarray*}\Vert a^\frac12 c\Vert^2&=&\Vert cac\Vert=\Vert c\langle TT^*x,x\rangle c\Vert=\Vert \langle TT^*(xc),(xc)\rangle\Vert\\
&\le& \Vert \langle SS^*(xc),(xc)\rangle\Vert=\Vert b^\frac12 c\Vert^2.
\end{eqnarray*}
Therefore, by Lemma~\ref{lem:lance p25}, we have $a\le b$, and thus
$$\langle TT^*x,x\rangle\le \langle SS^*x,x\rangle\ \mbox{for all $x \in K$},$$
which means, by Lemma~\ref{lem:positive operator}, that $TT^*\le SS^*$.
\end{proof}

A direct application of the preceding theorem is as follows:
\begin{corollary}\label{cor:equivalence of i and ii} Let $E, H$ and $K$ be Hilbert $\mathfrak{A}$-modules, $T\in {\mathcal L}(E,K)$ and $S\in {\mathcal L}(H,K)$. Then the following two statements are equivalent:
\begin{enumerate}
\item[{\rm (i)}] $T^\prime (T^\prime)^*\le \lambda TT^*$ for some $\lambda>0$;
\item[{\rm (ii)}]There exists $\mu>0$ such that $\Vert (T^\prime)^*z\Vert\le \mu \Vert T^*z\Vert$ for all $z \in K$.
\end{enumerate}
\end{corollary}

\begin{remark}\label{rem:construction of operator V} Let $E, H$ and $K$ be Hilbert $\mathfrak{A}$-modules, and let $T\in {\mathcal L}(E,K)$ and $T^\prime\in {\mathcal L}(H,K)$ be such that $T^\prime (T^\prime)^*\le \lambda_0 TT^*$ for some $\lambda_0>0$. Put
$$\alpha=\inf\{\lambda: T^\prime (T^\prime)^*\le \lambda TT^*\}.$$ Since $\mathcal{L}(K)_+$ is a closed cone, we have $T^\prime (T^\prime)^*\le \alpha TT^*=\left(\sqrt{\alpha}T\right)\left(\sqrt{\alpha}T\right)^*$. It follows from
Theorem~\ref{thm:main thm-1} that
\begin{equation}\label{equ:pre-well defined}\Vert (T^\prime)^*z\Vert\le \sqrt{\alpha}\, \Vert T^*z\Vert\ \mbox{for all $z \in K$}.\end{equation}
Put $V_1: \mathcal{R}(T^*)\to \mathcal{R}\left(\left(T^\prime\right)^*\right)$, $V_1(T^* z)=(T^\prime)^*z$ for any $z\in K$. Then
by \eqref{equ:pre-well defined} $V_1$ is well-defined and can be extended to be a bounded linear operator $V$ from
$\overline{\mathcal{R}(T^*)}\oplus \mathcal{N}(T)$ to $H$, such that $V\big|_{\mathcal{N}(T)}=0$ and $\Vert V\Vert\le \sqrt{\alpha}$. For any $z \in K$, it holds that
\begin{equation*}\Vert (T^\prime)^*z\Vert=\Vert V(T^*z)\Vert \le \Vert V\Vert \, \Vert T^*z\Vert
=\left\Vert \left(\Vert V\Vert T\right)^*z\right\Vert.\end{equation*}
Once again from Theorem~\ref{thm:main thm-1} we conclude that $T^\prime (T^\prime)^*\le \Vert V\Vert ^2 TT^*$. Hence
$\alpha\le \Vert V\Vert ^2$ and furthermore $\alpha=\Vert V\Vert ^2$. By definition we have $VT^*=(T^\prime)^*$, it may however happen that $T^\prime\ne T V^*$ since $V^*$ may not exist.
\end{remark}

\section{Characterization of the orthogonally complemented condition}\label{sec:orthogonally complemented condition}
Let $E, K$ be Hilbert $\mathfrak{A}$-modules and $T$ be in $\mathcal{L}(E,K)$. In this section, we study
the orthogonally complemented condition for $\overline{{\mathcal R}(T^*)}$, and the equivalence of conditions (iii) and (iv) in Theorem~\ref{thm:fang s result}.

\begin{lemma}\label{lem:closeness implies orthogonality}{\rm \cite[Theorem~3.2]{lance}} Let $E, K$ be Hilbert $\mathfrak{A}$-modules and suppose that $T \in \mathcal{L}(E,K)$ has closed range. Then $\mathcal{R}(T)$ is orthogonally complemented.
\end{lemma}

Now for any $T\in\mathcal{L}(E,K)$, let
 \begin{equation}\label{equ:defn of G and S}G=\overline{\mathcal{R}(T^*)} \ \mbox{and}\ S\ \mbox{be the restriction of $T$ on $G$.}\end{equation}
Clearly, ${\mathcal R}(S)\subseteq {\mathcal R}(T)$ and $S\in\mathcal{L}(G,K)$ which satisfies
\begin{equation}\label{equ:characterization of the star of S}S^*u=T^* u\ \mbox{for all $u \in K$}.\end{equation}

\begin{theorem} \label{thm:two orthogonally complemented conditions} Let $E, K$ be Hilbert $\mathfrak{A}$-modules and $T$ be in $\mathcal{L}(E,K)$. Then the following statements are equivalent:
\begin{enumerate}
\item[{\rm (i)}] $\overline{{\mathcal R}(T^*)}$ is orthogonally complemented;
\item[{\rm (ii)}] For any Hilbert $\mathfrak{A}$-module $H$ and any $T^\prime\in {\mathcal L}(H,K)$, the equation
$T^\prime =TX$ for $X\in {\mathcal L}(H,E)$ is solvable whenever ${\mathcal R}(T^\prime)\subseteq {\mathcal R}(T)$;
\item[{\rm (iii)}] The equation
$S=TX$ for $X \in {\mathcal L}(G,E)$ is solvable, where $G$ and $S$ are defined by \eqref{equ:defn of G and S}.
\end{enumerate}
If condition (i) is satisfied and $T^\prime\in {\mathcal L}(H,K)$ is such that ${\mathcal R}(T^\prime)\subseteq {\mathcal R}(T)$, then there exists a unique $D\in {\mathcal L}(H,E)$ which satisfies
\begin{equation}\label{equ:uniqueness of D}T^\prime=TD\ \mbox{and}\ {\mathcal R}(D)\subseteq {\mathcal N}(T)^\bot.\end{equation}
In this case,
\begin{equation}\label{equ:additional conditions of D}\Vert D\Vert^2=\inf\left\{\lambda : T^\prime(T^\prime)^*\le \lambda TT^*\right\}.\end{equation}
\end{theorem}
\begin{proof}``(i)$\Longrightarrow$(ii)" follows from the proof of \cite[Theorem~1.1]{fang-yu-yao}. ``(ii)$\Longrightarrow$(iii)" is obvious.

 ``(iii)$\Longrightarrow$(i)": Suppose that there exists $X\in {\mathcal L}(G,E)$ such that $S=TX$.
 Taking $*$-operation, we get $X^* T^*=S^*$. This, together with \eqref{equ:characterization of the star of S} and the definition of $G$ given by \eqref{equ:defn of G and S}, yields
\begin{equation}\label{equ:X star identity}X^*z=z\ \mbox{for all $z \in G$}.\end{equation}
Note that $X^*\in {\mathcal L}(E,G)$,
so by \eqref{equ:X star identity} we conclude that
$\mathcal{R}(X^*)=G$, which is closed. It follows from Lemma~\ref{lem:closeness implies orthogonality} that $G$ is orthogonally complemented.

Suppose that condition (i) is satisfied and ${\mathcal R}(T^\prime)\subseteq {\mathcal R}(T)$. By the proof of \cite[Theorem~1.1]{fang-yu-yao} we know that in this case, condition (ii) in Theorem~\ref{thm:fang s result} is satisfied and the operator $V$ defined in Remark~\ref{rem:construction of operator V}
is adjointable such that $V^*$ is the unique solution to \eqref{equ:uniqueness of D}. Furthermore, by Remark~\ref{rem:construction of operator V} we know that $\Vert V^*\Vert^2=\Vert V\Vert^2=\inf\{\lambda: T^\prime (T^\prime)^*\le \lambda TT^*\}$.
\end{proof}

\begin{remark}Let $T\in \mathcal{L}(E,K)$ be such that $\overline{{\mathcal R}(T^*)}$ is not orthogonally complemented. Then
from the preceding theorem we conclude that the equation
$S=TX$ for $X \in {\mathcal L}(G,E)$ is unsolvable, where $G$ and $S$ are defined by \eqref{equ:defn of G and S}. This kind of unsolvability may also happen for other Hilbert $\mathfrak{A}$-modules $H$ and certain operator $T^\prime\in {\mathcal L}(H,K)$. Such an example is as follows:
\end{remark}
\begin{proposition}\label{prop:nonequivalence} There exist Hilbert $\mathfrak{A}$-modules $H,K$, and $T\in {\mathcal L}(H,K)$ and $T^\prime\in {\mathcal L}(K)$ such that ${\mathcal R}(T^\prime)$ is contained in ${\mathcal R}(T)$ properly, whereas the equation $TX=T^\prime$ for $X\in\mathcal{L}(K,H)$ has no solution.
\end{proposition}
\begin{proof} Let $E$ be any countably infinite dimensional Hilbert space, $\mathbb{B}(E)$ (resp.\,$\mathbb{K}(E)$) be the set of all bounded (resp.\,compact) linear operators on $E$, and $I_E$ be the identity operator on $E$. Let $\mathfrak{A}=\mathbb{B}(E)$, $H=\mathbb{B}(E)$ and $K=\mathbb{K}(E)$. Then $H$ and $K$ are Hilbert $\mathfrak{A}$-modules whose $\mathfrak{A}$-valued inner-products are given by $\langle x,y\rangle=x^*y$ for any $x,y\in H$ and $x,y\in K$, respectively.

Choose any element $s \in K_+$ such that
$\overline{s\,K}=K$, where $sK=\{sy:y\in K\}$ and $\overline{s\,K}$ denotes the norm closure of $sK$ (see \cite[Example~3.1]{Xu-Fang} for the construction of such an element $s$).
Let $T\in\mathcal{L}(H,K)$ be defined by
\begin{equation}\label{equ:defn of T}T(x)=sx\ \mbox{for any $x\in H$}.\end{equation}
Then clearly, $T^*\in\mathcal{L}(K, H)$ and $T^\prime\stackrel{def}{=}(TT^*)^\frac12\in\mathcal{L}(K)$ are given by
\begin{equation}\label{equ:expression for T prime}T^*(y)=T^\prime(y)=sy\ \mbox{for any $y\in K$}.\end{equation}
As $I_E\in H$ and $K\subseteq H$, it follows from \eqref{equ:defn of T} and \eqref{equ:expression for T prime} that
$s\in \mathcal{R}(T)$ and $\mathcal{R}\left(T^\prime\right)\subseteq \mathcal{R}(T)$. We prove that $s\notin \mathcal{R}\left(T^\prime\right)$. Indeed, if $s=sy$ for some $y\in K$, then $y^*s=s$, which means by $\overline{s\,K}=K$ that $y^* z=z$ for all $z \in K$. This happens only if $I_E=y^*\in K$, in contradiction to the assumption that $E$ is infinite dimensional. This completes the proof that ${\mathcal R}(T^\prime)$ is contained in ${\mathcal R}(T)$ properly.

Assume that there exists $X\in\mathcal{L}(K,H)$ such that $TX=T^\prime$. Taking $*$-operation, we get
$X^* T^*=T^\prime$. This, together with \eqref{equ:expression for T prime} and the equality $\overline{s\,K}=K$, yields $X^*y=y$ for all $y \in K$. Thus, for any $u,v\in K$, we have
\begin{equation*}(Xu)^*v=\langle Xu,v\rangle=\langle u, X^*v\rangle=\langle u, v\rangle=u^*v,
\end{equation*}
which means that $(Xu)^*=u^*$, and hence $Xu=u$ for any $u\in K$.

Put $a=X^*(I_E)$. Then $a\in K$. For any $u\in K$, we have
\begin{equation*}u^*=\langle u, I_E\rangle=\langle Xu, I_E\rangle=\langle u, a\rangle=u^*a.
\end{equation*}
Taking $*$-operation, we get $a^* u=u$ for any $u\in K$, hence $a^*=I_E\in K$, which is a contradiction.
\end{proof}

\vspace{5ex}
\noindent\textbf{Acknowledgement}
\vspace{2ex}

This research was initiated while the second author visited the Department of Mathematics at Shanghai Normal University in 2017. The authors thank the referee for valuable suggestions.

\bibliographystyle{amsplain}

\end{document}